\documentclass[10pt]{article}
\usepackage[margin=2.3cm,a4paper]{geometry}
\usepackage{amsmath}
\usepackage{amssymb}
\usepackage{amsthm}
  \RequirePackage{amsfonts} 
\usepackage[utf8]{inputenc}
\usepackage{tikz,pgfplots}
\usetikzlibrary{patterns,intersections}
\usepgfplotslibrary{fillbetween}

\pgfplotscreateplotcyclelist{res}{%
{color=blue,solid,mark=+,mark options={solid,mark size=3pt} ,line width=1.0pt, mark repeat=5},
{color=orange,dashed,mark=x, mark options={solid,mark size=3pt},line width=1.0pt, mark repeat=5},
{color=blue, thick,dashdotted, mark=triangle, mark options={solid,mark size=3pt}, mark repeat=5},
{color=orange,dashed,line width=1.0pt, mark=o, mark options={solid,mark size=3pt},mark repeat=5},
{color=red, densely dotted, mark=diamond*, mark options={solid,mark size=3pt}, mark repeat=5},
{color=red, densely dashed, mark=*, mark repeat=4},
}

\usepackage{tabularx}
\usepackage{multirow}
\usepackage{fullpage}
\usepackage{rotating}
\usepackage{enumerate}

\usepackage{graphicx}  
\usepackage{caption,subcaption}

\usepackage[algo2e,vlined,longend,linesnumbered,ruled]{algorithm2e}
\newcommand{\smb}{\left[\begin{smallmatrix}}
\newcommand{\sme}{\end{smallmatrix}\right]}
\newcommand{\trp }{*}
\newcommand{\myparen}[1]{\left(#1\right)}
\DeclareMathOperator*{\rank}{rank}
\newcommand{\diag}[1]{\ensuremath{\operatorname{diag}\!\myparen{#1}}}

\newcommand{\half}{\ensuremath{\frac{1}{2}}}

\newcommand\R{\ensuremath{\mathbb{R}}}
\newcommand{\Rn}{\ensuremath{\R^{n}}}
\newcommand{\Rnn}{\ensuremath{\R^{n\times n}}}
\newcommand{\Rmm}{\ensuremath{\R^{m\times m}}}
\newcommand{\Rnm}{\ensuremath{\R^{n\times m}}}

\newcommand{\Rnr}{\ensuremath{\R^{n\times r}}}
\newcommand{\Rmr}{\ensuremath{\R^{m\times r}}}

\newcommand{\Rk}{\ensuremath{\R^{k}}}

\newcommand\C{\ensuremath{\mathbb{C}}}

\newcommand{\Cnn}{\ensuremath{\C^{n\times n}}}

\newcommand{\Cnr}{\ensuremath{\C^{n\times r}}}
\newcommand{\Cmr}{\ensuremath{\C^{m\times r}}}

\newcommand{\cA}{\ensuremath{\mathcal{A}}}
\newcommand{\cB}{\ensuremath{\mathcal{B}}}
\newcommand{\cC}{\ensuremath{\mathcal{C}}}
\newcommand{\cL}{\ensuremath{\mathcal{L}}}
\newcommand{\cO}{\ensuremath{\mathcal{O}}}
\newcommand{\cR}{\ensuremath{\mathcal{R}}}
\newcommand{\cW}{\ensuremath{\mathcal{W}}}

\newcommand{\fR}{\ensuremath{\mathfrak{R}}}

\newcounter{mymac@matlab}
  \newcommand{\matlab}{MATLAB%
   \ifnum\value{mymac@matlab}<1%
   \textsuperscript{\textregistered}%
   \setcounter{mymac@matlab}{1}%
   \fi%
  }
	\newcommand{\intel}{Intel\textsuperscript{\textregistered}}
  \newcommand{\coretwo}{Core\texttrademark 2}

\newcommand{\kmax}{\ensuremath{k_{\max}}}

\usepackage{mathdots}
\usepackage[colorlinks=true,citecolor = blue]{hyperref}
\usepackage[space,noadjust,nocompress]{cite}

\newtheorem{defin}{Definition}
\newtheorem{theorem}{Theorem}
\newtheorem{remark}{Remark}

\newtheorem{corollary}[theorem]{Corollary}

\hypersetup{%
pdftitle = {Inexact linear solves in the low-rank ADI iteration for large Sylvester equations}, %
pdfsubject = {Numerical Methods for Large Matrix Equations}, %
pdfauthor = {Patrick K\"urschner}, %
pdfkeywords = {Sylvester equation, alternating
direction implicit, rational Krylov subspaces, inner-outer methods}
}
\title{Inexact linear solves in the low-rank ADI iteration for large Sylvester equations} %
\date{}
\author{Patrick K\"{u}rschner\thanks{Leipzig University of Applied Sciences (HTWK Leipzig), Center for Mathematics and Natural Sciences,
 Leipzig, Germany, {\tt
  patrick.kuerschner@htwk-leipzig.de}} }

\begin{document}
\maketitle

\begin{abstract}
We consider the low-rank alternating directions implicit (ADI) iteration for approximately solving large-scale algebraic Sylvester equations.
Inside every iteration step of this iterative process a pair of linear systems of equations has to be solved. We investigate the situation when those inner linear systems are solved inexactly by an iterative methods such as, for example, preconditioned Krylov subspace methods. The main contribution of this work are thresholds for the required accuracies regarding the inner linear systems which dictate when the employed inner Krylov subspace methods can be safely terminated. 
The goal is to save computational effort by solving the inner linear system as inaccurate as possible without endangering the functionality of the low-rank Sylvester-ADI method. Ideally, the inexact ADI method mimics the convergence behaviour of the more expensive exact ADI method, where the linear systems are solved directly. Alongside the theoretical results, also strategies for an actual practical implementation of the stopping criteria are developed. Numerical experiments confirm the effectiveness of the proposed strategies.
\end{abstract}
	
\section{Introduction}
We consider the numerical solution of large scale algebraic Sylvester equations of the form  
\begin{align}\label{sylv}
 AX+XB=-fg^\trp
\end{align}
with large coefficient matrices $A\in\Rnn,~B\in\Rmm$, the sought after matrix $X\in\Rnm$, and a factorized right hand side with $f\in\Rnr,g\in\Rmr$ which have full column rank $r\ll n,m$.
Sylvester equations play a vital role in many applications areas, for instance, control theory, model order reduction~\cite{Antoulas05,SorA02}, image processing~\cite{CalR96}, and they often arise as crucial ingredient in algorithms in more complicated matrix equations~\cite{BreSS16,Jaretal18,BreS15,PalK2021}. We refer to the survey articles \cite{BhaR97,Simoncini16} for further details and examples.

For equations defined by small to medium sized 
coefficient matrices, direct methods based on matrix factorizations can be used such as the Bartels-Stewart and related methods \cite{BarS72}.
When only one coefficient matrix ($A$ or $B$) is large and sparse, while the other one is small and dense, 
special methods are applicable~\cite{SorA02,GolV13}.

We are interested here in the case when both coefficients, $A$ and $B$, are large and sparse matrices and the right hand side is of low-rank. For this scenario, one can show that the singular values of the solution $X$ typically decay rapidly towards zero~\cite{Gra04, Sab07, Clo23}. This motivates to compute a solution approximation of low-rank in factored form $X\approx Z\Gamma Y^\trp$, where $Z,Y$ are thin rectangular matrices and a middle matrix $\Gamma$ of a appropriate size.
In the recent years, different (rational) Krylov subspace projection methods have been proposed for this purpose, e.g., \cite{HuR92,GueJR02,BreSS16,PalS18,KresLMP21}. Another method for this situation, and the focus of this study, is the low-rank version of the alternating
directions implicit (ADI) iteration \cite{Wac88,CalR96,BenLT09,Wac13,BenK14}. 

Inside every step of the Sylvester ADI iteration, linear systems of equations have to be solved. Solving these inner linear systems is computationally to most expensive part of every iteration step and if direct solvers are not applicable, a further, inner iteration can be used in order to obtain approximate solutions. Typically one uses preconditioned Krylov subspace methods for this purpose. We develop estimates for the required accuracies regarding those linear systems which will dictate when the inner Krylov subspace methods can be safely terminated, thus potentially saving some computational effort without endangering the functionality of the low-rank Sylvester ADI method. A speciality of the Sylvester ADI is that a pair of two different linear systems has to be solved in every step, defined by shifted versions of the coefficients $A$ and $B$. Hence, one has to determine two connected stopping criteria, for which strategies will be presented.

The paper is structured as follows: in Section~\ref{sec:adi} we review the Sylvester ADI iteration and its properties, which we then modify to incorporate inexact solutions of the inner linear systems. Afterwards in Section~~\ref{sec:inner}, we develop dynamic inner stopping criteria for the iterative solution of the arising linear systems and analyse their effect on the Sylvester ADI iteration. We also discuss the actual implementation of these stopping criteria inside the ADI iteration. Section~~\ref{sec:num} demonstrates the findings by some numerical experiments and Section~~\ref{sec:concl} concludes and gives some potential future research directions.

\subsection{Notation}
Throughout the paper, if not stated otherwise, we use $\|\cdot\|$ to denote the Euclidean vector and associated induced matrix norm. $(\cdot)^\trp$ stands for the transpose for real and, respectively, complex conjugate transpose for complex matrices and vectors. The identity matrix of dimension $k$ is denoted by $I_k$ with the subscript omitted if the dimension is clear from the context. The $k$th column of the identity matrix is $e_k$ and $\boldsymbol{1}_k:=[1,\ldots,1]^\trp\in\Rk$. The spectrum of a matrix $A\in\Cnn$ and a matrix pair $(A,M)$ is denoted by $\Lambda(A)$ and $\Lambda(A,M)$, respectively. The smallest (largest) eigenvalue of a matrix $X$ are denoted by $\lambda_{\min}(A)$ ($\lambda_{\max}(A)$), $\rho(X)=\max\lbrace|\lambda|,~\lambda\in\Lambda(X)\rbrace$ is the spectral radius, and $\cW(A)=\lbrace z=x^\trp Ax:~0\neq x\in\Rn,~\|x\|=1\rbrace$ is the field of values. The symbol $\otimes$ denotes the Kronecker product. 
For $A\in\Cnn$, $\beta\in\C$, and $\alpha\notin-\Lambda(A)$, 
a two-parameter Cayley transformation is given by
\begin{align}\label{cayley}
\cC(A,\alpha,\beta)=(A-\beta I_n)(A+\alpha I_n)^{-1}=I_n-(\beta+\alpha)(A+\alpha I_n)^{-1}.
\end{align}
\section{The inexact low-rank ADI iteration for Sylvester equations}\label{sec:adi}
The Sylvester equation~\eqref{sylv} has a unique solution, if $\Lambda(A)\cap\Lambda(-B)=\emptyset$. This is, in particular, fulfilled if $\Lambda(A),\Lambda(B)\subset\C_-$ (or $\Lambda(A),\Lambda(B)\subset\C_+$) which is assumed in the remainder.  
\subsection{Derivation and review of basic properties}
For every $\beta\notin\Lambda(A),~\alpha\notin\Lambda(B)$,
$\alpha\neq\beta$ the continuous-time Sylvester equation~\eqref{sylv} is equivalent to the discrete-time Sylvester equation
\begin{align}\label{discrete_sylv}
 X=\cC(A,\beta,\alpha)X\cC(B,\alpha,\beta)+T(\alpha,\beta),
\end{align}
where
\begin{align}\label{Tab}
T(\alpha,\beta):=-(\beta+\alpha)(A+\beta I_n)^{-1}fg^\trp(B+\alpha I_m)^{-1}
\end{align}
and $\cC(\cdot,~\cdot,~\cdot)$ are two-parameter Cayley transformations~\eqref{cayley}.  
 This motivates the instationary iteration for $k=1,2,\ldots$
\begin{align}\label{sylvadi}
\begin{split}
 X_{k}&=\cC(A,~\beta_k,~\alpha_k)X_{k-1}\cC(B,~\alpha_k,~\beta_k)+T(\alpha_k,\beta_k)\\
&=\cC_k(A)X_{k-1}\cC_k(B)+T_k,
\end{split}
\end{align}
where we introduced varying parameters $\alpha_k,~\beta_k$ throughout the iteration. This is the alternating directions implicit iteration (ADI) for algebraic Sylvester equations~\cite{Wac88}.
The error and the Sylvester residual matrix after $k$ steps of the Sylvester ADI scheme~\eqref{sylvadi}
are given by
\begin{align}\label{fadi_error}
 X_k-X&=\cA_k(X_0-X)\cB_k,\quad \cR_k=AX_k+X_kB-fg^\trp=\cA_k\cR_0\cB_k,\\\nonumber
\cA_k&:=\prod\limits_{i=1}^k\cC(A,\beta_i,
\alpha_i),\quad \cB_k:=\prod\limits_{i=1}^k\cC(B,\alpha_i,\beta_i).
\end{align}
The iteration is convergent, e.g., when $\rho(\cC_k(A))\rho(\cC_k(B))<1$ for all $k\geq 1$. 
A rapid decrease of error and residual can thus be achieved by minimizing this product of the spectral radii of $\cA_k$ and $\cB_k$ which leads to the ADI shift parameter problem
\begin{align}\label{adi_shift}
\min_{\alpha_i,\beta_i\in\C}\Bigg(\mathop{\max_{1\leq\ell\leq n}}_{1\leq j\leq
m}
\prod\limits_{i=1}^k\left|\frac{
(\lambda_{\ell}-\alpha_i)(\mu_j-\beta_i) } {
(\lambda_{\ell}+\beta_i)(\mu_j+\alpha_i)}\right|\Bigg),\quad
\lambda_{\ell}\in\Lambda(A),~\mu_j\in\Lambda(B).
\end{align} 
Various approaches have been proposed for~\eqref{adi_shift}, e.g., precomputing shift parameters \cite{Sab07,BenLT09,Wac13} in an offline-phase before the actual ADI iteration, either based on elliptic functions regions or heuristic strategies. In contrast, more recent developments were focussed at computing one pair $(\alpha_k,\beta_k)$ of parameters at a time online during the running ADI iteration~\cite{BenKS14,Kue16}. In this study, we assume that we the pairs $(\alpha_k,\beta_k)$ which guarantee convergence of~\eqref{sylvadi} are given in advance. This is purely for reasons of simplification, as the upcoming results on the stopping criteria do not depend on the way the shifts are generated.

A low-rank version\footnote{This version is also called factored ADI iteration (fADI).} of the ADI iteration~\cite{BenLT09} is obtained by setting $X_0=0$ in~\eqref{sylvadi} and exploiting that the matrices $(A+\beta_k I)^{-1}$ and $(A-\alpha_k I)$ as well as $(B+\alpha_k I)^{-1}$ and $(B-\beta_k I)$ commute. This allows to formulate~\eqref{sylvadi} for $k\geq 1$ as
\begin{subequations}\label{fadi_iter}
 \begin{align}\label{fadi_iter_a}
 z_1&=(A+\beta_1I_n)^{-1}f,\quad 
s_1=(B+\alpha_1I_m)^{-H}g,\\\label{fadi_iter_b}
z_k&=z_{k-1}+(\beta_k-\alpha_{k-1})(A+\beta_kI_n)^{-1}z_{k-1},\\
\label{fadi_iter_c}
y_k&=y_{k-1}+\overline{(\alpha_k-\beta_{k-1})}(B+\alpha_kI_m)^{-H}y_{k-1}
\end{align}
which produces solution approximations in low-rank format
 \begin{align*}
 X\approx X_{k}&=Z_k\Gamma_kY_k^*\quad\text{with}\quad
Z_k=[z_1,\ldots,z_k]\in\C^{n\times kr},\quad \\\Gamma_{k}&=\diag{(\Gamma_{k-1},(\beta_{k} -\alpha_{k})I_r)}\in\C^{kr\times kr},\quad Y_k=[y_1,\ldots,y_k]\in\C^{m\times kr}.
\end{align*}
\end{subequations}
It can be shown~\cite{BenK14,Kue16} that the residual matrix~\eqref{fadi_error} at step $k$ has at most rank $r$ and is given by the low-rank factorization
\begin{subequations}\label{adi_resexact}
\begin{align}
 \cR_k&=w_kt_k^*,\quad\text{where}\\
w_k&:=\cA_kf=w_0+Z_k\boldsymbol{\gamma}_k=w_{k-1}+\gamma_k(A+\beta_kI)^{-1}w_{k-1}\in\Cnr,\\
t_k&:=\cB^*_kg=t_0+Y_k\overline{\boldsymbol{\gamma}_k}=t_{k-1}+\overline{\gamma_k}(B+\alpha_kI)^{-*}t_{k-1}\in\Cmr,\\
\text{with}\quad w_0&:=f,\quad t_0:=g,\quad \boldsymbol{\gamma}_k:=[\gamma_1,\ldots,\gamma_k]^\trp\otimes I_r,\quad \gamma_k:=-(\beta_k+\alpha_k).
\end{align}
\end{subequations}
 This allows to compute the residual norm $\|\cR_k\|_2=\|w_kt_k^*\|_2$ more efficiently and the residual factors $w_k,~t_k$ can be directly integrated into the low-rank Sylvester ADI iteration~\eqref{fadi_iter} which then becomes
\begin{subequations}\label{lradi_basic}
 \begin{align}\label{lradi_1}
	z_k	&=(A+\beta_k I)^{-1}w_{k-1},\quad w_k=w_{k-1}+\gamma_kz_k,\quad w_0:=f\\\label{lradi_2}
	y_k	&=(B+\alpha_k I)^{-*}t_{k-1},\quad t_k=t_{k-1}+\overline{\gamma_k}y_k,\quad t_0:=g,
\end{align}
\end{subequations}
which is the form of the iteration used nowadays.
\paragraph{Generalized Sylvester equations}
For two given additional nonsingular matrices $M\in\Rnn$ and $C\in\Rmm$, generalized Sylvester equations are of the form
\begin{align}\label{gsylv}
 AXC+MXB=-fg^\trp.
\end{align}
The low-rank ADI iteration can be generalized in a straightforward manner to
\begin{align}\label{gadi}
\begin{split}
 X_{k}&=\cC(AM^{-1},\beta_k,\alpha_k)X\cC(C^{-1}B,\beta_k,\alpha_k)+\hat T_k\quad\text{with}\\
\hat T_k&:=-\gamma_k(A+\beta_k M)^{-1}fg^\trp(B+\alpha_k C)^{-1}.
\end{split}
\end{align}

The corresponding low-rank iteration~\eqref{lradi_basic} becomes in that case~\cite{BenK14,Kue16}
\begin{subequations}\label{lradi_g}
 \begin{align}\label{lradi_gA}
	z_k	&=(A+\beta_k M)^{-1}w_{k-1},\quad w_k=w_{k-1}+\gamma_kMz_k,\quad w_0:=f,\\\label{lradi_gB}
	y_k	&=(B+\alpha_k C)^{-*}t_{k-1},\quad t_k=t_{k-1}+\overline{\gamma_k}C^*y_k,\quad t_0:=g.
\end{align}
\end{subequations}
The spectra $\Lambda(A)$ and $\Lambda(B)$ in~\eqref{adi_shift} have to be replaced by the spectra $\Lambda(A,M)$ and $\Lambda(B,C)$, respectively.
\begin{remark}
Note that
 \begin{align}\label{lradi_g_wt}
\begin{split}
	w_k&=w_{k-1}+\gamma_kM(A+\beta_k M)^{-1}w_{k-1}=(A-\alpha_kM)(A+\beta_k M)^{-1}w_{k-1}\\
	&=(AM^{-1}-\alpha_k I)(AM^{-1}+\beta_k I)^{-1}w_{k-1}=\cC_k(AM^{-1})w_{k-1},\\
	t_k&=t_{k-1}+\overline{\gamma_k}C^*(B+\alpha_k C)^{-*}t_{k-1}=(B-\beta_k C)^*(B+\alpha_k C)^{-*}t_{k-1}\\
	&=(C^{-1}B-\beta_k I)^*(C^{-1}B+\alpha_k I)^{-*}t_{k-1}=\cC^*_k(C^{-1}B)t_{k-1}
	\end{split}
\end{align}
indicating that the matrices of $AM^{-1}$, $C^{-1}B$ appear only for notational purposes and will not formed explicitly in an actual implementation.
\end{remark}
We will mostly consider the more general version~\eqref{lradi_g} in the remainder.

The most expensive parts inside each step of the outer ADI iteration~\eqref{lradi_basic},\eqref{lradi_g} are the solutions of the inner shifted linear system with
$(A+\beta_kM)$ and $(B+\alpha_k C)^\trp$ with $r$ right hand sides for $z_k$ and $y_k$. 

Allowing inexact solutions of these linear systems but keeping all remaining steps in~\eqref{lradi_basic},\eqref{lradi_g} unchanged
gives the \emph{inexact low-rank ADI iteration} illustrated in Algorithm~\ref{alg:lradi}. 
Let
\begin{align}\label{linres}
\begin{split}
r^A_k&:=w_{k-1}-(A+\alpha_k M)\tilde z_k\quad\text{with}\quad\|r^A_k\|\leq \delta^{A}_k,\\
r^B_k&:=t_{k-1}-(B+\beta_k C)^\trp\tilde y_k\quad\text{with}\quad\|r^B_k\|\leq \delta^{B}_k
\end{split}
\end{align}
be the residual vectors with respect to the linear systems and approximate solutions $\tilde z_{k}$, $\tilde y_{k}$. We will often refer to $r^A_k,~r^B_k$ as inner residuals. The quantities $\delta^{A}_k,~\delta^{B}_k>0$ indicate the residual tolerances with respect to the inner linear systems at step $k$ of the inexact LR-ADI iteration.

\begin{algorithm2e}[t]
\SetEndCharOfAlgoLine{}
\SetKwInOut{Input}{Input}\SetKwInOut{Output}{Output}
  \caption[Inexact low-rank ADI for Sylvester eqns.]{Inexact low-rank ADI (LR-ADI) method for Sylvester equations}
  \label{alg:lradi}
    \Input{$A,~B,~M,~C,~f,~g$ as in \eqref{gsylv}, 
    shift parameters $\lbrace \alpha_1,\dots,\alpha_{k}\rbrace$, $\lbrace
\beta_1,\dots,\beta_{k}\rbrace$, and tolerance $0<\tau\ll1$.}
    \Output{$Z_{k}\in\C^{n\times rk}$,
$Y_{k}\in\C^{m\times rk}$, $\Gamma_{k}\in\C^{rk\times rk}$ such that
    $Z_{k}\Gamma_{k}Y_{k}^H\approx X$.}
 $w_0=f,~t_0=g$, $Z_0=\Gamma_0=Y_0=[~]$, $j=1$.\;
    \While{$\|w_{k-1}T^H_{k-1}\|\geq\tau\|FG^\trp\|$}{%
Approximately solve the linear systems for $z_k,~y_k$:
\begin{align}\label{gfadi_incsolves}
(A+\beta_kM)z_k&=w_{k-1},\quad\|r^A_k\|=\|w_{k-1}-(A+\beta_kM)z_k\|\leq\delta^A_k,\\
(B+\alpha_kC)^*y_k&=t_{k-1},\quad\|r^B_k\|=\|t_{k-1}-(B+\alpha_kC)^*y_k\|\leq\delta^B_k.
\end{align}\;
$\gamma_k:=-(\beta_k+\alpha_k)$.\;
$w_k=w_{k-1}+\gamma_kMz_k,\quad$
$t_k=t_{k-1}+\overline{\gamma_k}C^*y_k$.\nllabel{lradi_resup}\;
$Z_{k}=[Z_{k-1},~z_k],~Y_{k}=[Y_{k-1},~y_k],
~\Gamma_{k}=\diag{\Gamma_{k-1},~\gamma_kI_r}$.\;
$k=k+1$.\;
}
\end{algorithm2e}
\begin{remark}
In this work, 'inexact' means that the solution process of the linear systems is the only source of inexactness in the LR-ADI iteration. We do not consider other errors introduced, e.g., by the finite-precision arithmetic.
\end{remark}
\begin{remark}
 An inexact version of the dense iteration~\eqref{sylvadi} for positive definite $A,B$, $M=I_n,~C=I_m$ and fixed shifts $\alpha_k=\alpha,~\beta_k=\beta,~\forall~k\geq 1$ is discussed in~\cite{LZZ20}. The motivation in~\cite{LZZ20} is to ensure asymptotic convergence of~\eqref{sylvadi} under inexact inner solves. In contrast, our analysis is focused on the low-rank iteration with variable shifts and is, moreover, based on the goal to make the behaviour of the inexact ADI iteration as close as possible to that of the exact iteration. Note that by our assumptions on the shift parameters, the exact ADI iteration converges.
\end{remark}
\subsection{Properties of the inexact low-rank Sylvester ADI iteration}
Before we can pursue stopping criteria for the inexact low-rank Sylvester ADI iteration, we need to generalize some results for the exact iteration~\cite[Corollary~3.16]{Kue16} as well as for the inexact iteration for Lyapunov equations~\cite[Theorem~3.2]{KueF20}.
\begin{theorem}
\label{th:adi_AZ_BTY}
The low-rank solution factors
$Z_k$, $Y_k$ constructed after $j$ steps of the inexact LR-ADI iteration
(Algorithm~\ref{alg:lradi}) 
satisfy the identities
\begin{align}\label{adi_AZ_BTY}
 AZ_{k}&=MZ_{k}\boldsymbol{\sigma}^\alpha_{k}+w_kE_k^\trp-S_k^A,
 &B^\trp Y_{k}&=C^\trp Y_{k}\boldsymbol{\sigma}^\beta_{k}+t_kE_k^\trp-S_k^B\\\nonumber
\text{with}\quad\boldsymbol{\sigma}^\alpha_{k}&
:=
\left[\begin{smallmatrix}
 \alpha_1&&&\\
-\gamma_2&\alpha_2&&\\
\vdots&&\ddots&\\
-\gamma_k&\cdots&-\gamma_k&\alpha_k
\end{smallmatrix}\right]\otimes I_r,
\quad
&\boldsymbol{\sigma}^\beta_{k}&:=
\left[\begin{smallmatrix}
 \overline{\beta_1}&&&\\
-\overline{\gamma_2}&\overline{\beta_2}&&\\
\vdots&&\ddots&\\
-\overline{\gamma_k}&\cdots&-\overline{\gamma_k}&\overline{\beta_k}
\end{smallmatrix}\right]\otimes I_r\in\C^{jr\times jr},\\\nonumber
E_k&:=\boldsymbol{1}_k\otimes I_r\in\R^{jr\times r},&&\\\nonumber
S_k^A&:=[r_1^A,\ldots,r_k^A]\in\C^{n\times rk},\quad &S_k^B&:=[r_1^B,\ldots,r_k^B]\in\C^{m\times rk}
\end{align}
contain the residuals of the linear systems~\eqref{linres}.
\end{theorem}
\begin{proof}
 The result for $S^A_k=0$, $S^B_k=0$ has been established in~\cite{Kue16}. 
 By construction, it holds $Az_i=w_i+\alpha_i Mz_i-r_i^A$ and $w_i=w_k-M\sum\limits_{j=1}^{k-i}\gamma_{k-j+1}z_{k-j+1}$ for
$i=1,\ldots,k$. Hence,
\begin{align*}
 Az_i=w_k+\alpha_i Mz_i-M\sum\limits_{j=1}^{k-i}\gamma_{k-j+1}z_{k-j+1}-r_i^A,\quad
i=1,\ldots,k
\end{align*}
such that
\begin{align*}
 A[z_1,\ldots,z_k]=[w_k,\ldots,w_k]+M[z_1,\ldots,z_k]
\left(\left[\begin{smallmatrix}
 \alpha_1&&&\\
-\gamma_2&\alpha_2&&\\
\vdots&&\ddots&\\
-\gamma_k&\cdots&-\gamma_k&\alpha_k
\end{smallmatrix}\right]\otimes I_r\right)-S^A_k
\end{align*}
which is the result in~\eqref{adi_AZ_BTY} for $AZ_k$. The result for $B^\trp Y_k$ is developed in the same way using the corresponding relations for $t_k,~y_k,~r_k^B$.
\end{proof}
\begin{theorem}
After $k$ iteration steps of the inexact LR-ADI iteration applied to~\eqref{gsylv}, 
the Sylvester residual matrix is given by
\begin{align}\label{adi_res_inexact}
\begin{split}
 \cR_k&=AZ_k\Gamma_kY_k^*C+MZ_k\Gamma_kY_k^*B+fg^*=w_kt_k^*-\eta^A_k-\eta^B_k,\\
\eta^A_k&:=S^A_k\Gamma_kY_k^*C,\quad \eta^B_k:=MZ_k\Gamma_k(S_k^B)^*.
\end{split}
\end{align}
\end{theorem}
\begin{proof}
By construction, it follows from~\eqref{adi_resexact} that $f=w_k-MZ_k\boldsymbol{\gamma}_k$ and $g=t_k-C^*Y_k\overline{\boldsymbol{\gamma}_k}$. 
Plugging this and the identities of Theorem~\ref{th:adi_AZ_BTY} into the Sylvester residual matrix yields 
\[
\cR_k=MZ_k\left(\boldsymbol{\sigma}^\alpha_{k}\Gamma_k+\Gamma_k(\boldsymbol{\sigma}^\beta_{k})^*+\boldsymbol{\gamma}_k\boldsymbol{\gamma}_k^\trp\right)Y_k^*C+w_kt_k^*-MZ_k\Gamma_k(S_k^B)^*-S^A_k\Gamma_kY_k^*C.
\]
The claim follows upon realizing that $\Gamma_k$ is the solution of the Sylvester equation $\boldsymbol{\sigma}^\alpha_{k}\Gamma+\Gamma(\boldsymbol{\sigma}^\beta_{k})^*+\boldsymbol{\gamma}_k\boldsymbol{\gamma}_k^\trp=0$.
\end{proof}
The above properties show that in the presence of inexact linear solves, there is a discrepancy between in the computed Sylvester residuals $\cR^{\text{comp.}}_k = w_kt_k^*$ and the true residuals $\cR_k^{\text{true}}$ given by \eqref{adi_res_inexact}. 
\begin{defin}[Residual gap]\label{def:resgapADI}
The residual gap after $k$ iteration steps of the low-rank Sylvester ADI is defined as
\[
\Delta \cR_k:=\cR_k^{\text{comp}}-\cR^{\text{true}}_k=S^A_k\Gamma_kY_k^*C+MZ_k\Gamma_k(S_k^B)^*=\eta^A_k+\eta_k^B
\]
\end{defin}
Hence, $\|\cR_k^{\text{comp}}\|=\|w_kt_k^*\|$ is not the correct value of the norm of the Sylvester residual for the current solution approximation $Z_k\Gamma_kY_k^*$. Using it might give a wrong impression of the iteration's progress.
\section{Dynamic residual thresholds for the inner linear systems}\label{sec:inner}
In this section we investigate strategies to make the inner residual norms $\|r^A_k\|,~\|r^B_k\|$ as large as possible without endangering the functionality of the outer ADI iteration. We plan to achieve that the inexact low-rank ADI iteration mimics the exact counterpart up to a very small deviation. 

Similar to the low-rank ADI iteration for Lyapunov equations, it can be easily shown that the low-rank factors $Z_k$ and $Y_k$ span (block) rational Krylov subspaces for $AM^{-1}$ and, respectively, $C^{-\trp}B^\trp$~\cite{LiW02}. Hence, the low-rank ADI iteration can be seen as a (two-sided) rational Krylov subspace method.
A commonly used approach in inexact (rational) Krylov and related methods is to enforce a small norm of the residual gap: $\|\Delta\cR_k\|\leq \varepsilon$. If $\|\cR^{\text{comp.}}_k\|\leq \varepsilon$ then a small true residual norm $\|\cR_k^{\text{true}}\|=\|\cR^{\text{comp.}}_k+\Delta\cR_k\|\leq 2\varepsilon$ can be expected. 
This technique is often coined \textit{relaxation} because it usually leads to increasing (relaxed) inner residual norms if the outer residual norm decrease~\cite{Sim05,BouF05}.
The next theorem generalizes \cite[Thm. 3.3]{KueF20} and provides a theoretical strategy for reducing $\|\Delta\cR_k\|$ in the inexact low-rank Sylvester ADI iteration below a desired threshold after a pre-specified number of outer steps.
\begin{theorem}[Theoretical inner stopping criteria the inexact Sylvester-ADI]\label{thm:ADI_theorelax}
Let the residual gap be given by Definition~\ref{def:resgapADI} with $w_k$, $t_k$, $\gamma_k$ as in~\eqref{lradi_g}. Let $\kmax$ be the maximum number of steps of Algorithm~\ref{alg:lradi} and $0<\varepsilon<1$ a small threshold. Furthermore, let
$c^A_k:=\|\cC_k(AM^{-1})\|$, $c^B_k:=\|\cC_k(C^{-1}B)\|$ and $\check{c}_k:=\max(c_k^A,~c_k^B)+1$.
If for $1\leq k\leq \kmax$, the linear system residuals satisfy
\begin{align}\label{ADI_relax1_theo}
		\check{c}_k(\|r^A_k\|\|t_{k-1}\|+\|r^B_k\|\|w_{k-1}\|+2\|r^B_k\|\|r^A_k\|) \leq \frac{\varepsilon}{\kmax}
\end{align} 
then $\|\Delta\cR_{\kmax}\|\leq\varepsilon$. 
\end{theorem}
\begin{proof}
Consider the following estimate 
\begin{align}\label{ADIresgapbound1}
	\|\Delta\cR_{\kmax}\|&\leq\sum\limits_{k=1}^{\kmax}\|r^B_k\|\|\gamma_kMz_k\|+\|r^A_k\|\|\overline{\gamma_k}C^*y_k\|.
\end{align}
Moreover, we can bound
\begin{align*}
\|\gamma_kMz_k\|&=\|\gamma_kM(A+\beta_k M)^{-1}(w_{k-1}-r^A_k)\|=\|(\cC_k(AM^{-1})-I_n)(w_{k-1}-r^A_k)\|\\
&\leq (c^A_k+1)(\|w_{k-1}\|+\|r^A_k\|),\\
\|\overline{\gamma_k}C^*y_k\|&=\|\overline{\gamma_k}C^*(B+\alpha_k C)^{-*}(t_{k-1}-r^B_k)\|=\|(\cC^\trp_k(C^{-1}B)-I_m)(t_{k-1}-r^B_k)\|\\
&\leq (c^B_k+1)(\|t_{k-1}\|+\|r^B_k\|),
\end{align*}
and  get
\begin{align*}
\|\Delta\cR_{\kmax}\|&\leq\sum\limits_{k=1}^{\kmax}(c^A_k+1)\|r^B_k\|(\|w_{k-1}\|+\|r^A_k\|)+(c^B_k+1)\|r^A_k\|(\|t_{k-1}\|+\|r^B_k\|)\\
&\leq\sum\limits_{k=1}^{\kmax}\check{c}_k(\|r^B_k\|\|w_{k-1}\|+\|r^A_k\|\|t_{k-1}\|+2\|r^A_k\|\|r^B_k\|) \leq\sum\limits_{k=1}^{\kmax}\frac{\varepsilon}{\kmax}=\varepsilon.
\end{align*}
if~\eqref{ADI_relax1_theo} holds. 
\end{proof}
\paragraph{Discussion and Consequences}
We already observe one striking difference compared to the Lyapunov situation: we do not get a single combination for the largest possible inner residual norms but instead infinitely many admissible combinations of $\|r^A_k\|,~\|r^B_k\|$ which satisfy \eqref{ADI_relax1_theo}.
We introduce the following notation:
 \[
\check{\varepsilon}_k:=\frac{\varepsilon}{\kmax}\quad\text{from}\quad\eqref{ADI_relax1_theo}.
\]
Due to the non-negativity of the residual norms, the bounds~\eqref{ADI_relax1_theo} dictate that admissible values are from the region bounded by
\begin{align}\label{relax_AB}
\begin{split}
0\leq\|r_k^A\|&\leq\frac{\check{\varepsilon}_k}{\check{c}_k\|t_{k-1}\|},\quad 0\leq\|r_k^B\|\leq\frac{\check{\varepsilon}_k}{\check{c}_k\|w_{k-1}\|}\quad\text{and}\quad \psi(\|r_k^A\|,\|r_k^B\|)\leq \check{\varepsilon}_k,\\
\psi(x,y)&:=\check{c}_k\left(x(\|t_{k-1}\|+y)+y(\|w_{k-1}\|+x)\right).
\end{split}
\end{align}
The largest possible inner residuals norms are located on the planar curve 
\[
0=\Psi(\|r_k^A\|,\|r_k^B\|):=\psi(\|r_k^A\|,\|r_k^B\|)-\check{\varepsilon}_k
\]
which is illustrated in Figure~\ref{fig:res_admiss}. How to select one particular combination will be discussed later.
\begin{figure}[t]
\begin{tikzpicture}[/pgf/declare function={f=(1-x)/(2/3*0.8*x+1);}]

\begin{axis}[%
width=0.7\linewidth, 
height=0.5\linewidth,
axis lines=middle,
xmin=-0.2,
xmax=1.3,
ymin=-0.2,
ymax=1.3,
xlabel={$\|r^A\|$},
ylabel={$\|r^B\|$},
restrict y to domain=-0.2:1.2,
xtick={1},
ytick={0,1},
xticklabels={$\frac{\check{\varepsilon}}{\check{c}\|t\|}$},
yticklabels={0,$\frac{\check{\varepsilon}}{\check{c}\|w\|}$},
extra x ticks={0},
extra x tick labels={0},
x tick label style={xshift={0.15em}},
axis background/.style={fill=white},
legend style={draw=black,fill=white,font=\small,legend cell align=left,legend columns=1,anchor=north west,at={(0.9,1.0)}},
legend entries={{$\Psi(\|r^A\|,\|r^B\|)=0$}, admissible tolerances, largest admissible tolerances}
]
\addplot [color=blue,thick,domain=-0.2:1.2,samples=100]
   {f};
	\addplot+[mark=none,
	draw=none,
	area legend,
        domain=0:1,
        samples=100,
        pattern=north east lines,
        pattern color=black!50]
				{f} \closedcycle;
				

\addplot [color=blue,dashed,line width=3pt,domain=0:1,samples=100,opacity=0.5]
   {f};	
	
\addplot[dotted, thin,black] coordinates {(1,0) (1,1)};
\addplot[dotted, thin,black] coordinates {(0,1) (1,1)};		
	
\end{axis}

\end{tikzpicture}%
\caption{Illustration of region of admissible inner residual norms with the boundary curve $\Psi=0$, where the thick dashed portion is the set of admissible combinations of largest tolerances. The ADI iteration indices were left out for readability. The data was generated here with $\check{c}=3$, $\check{\varepsilon}=10^{-8}$, $\|t\|=10^{-5}$, $\|w\|=2\cdot10^{-3}$.}
\label{fig:res_admiss}
\end{figure}
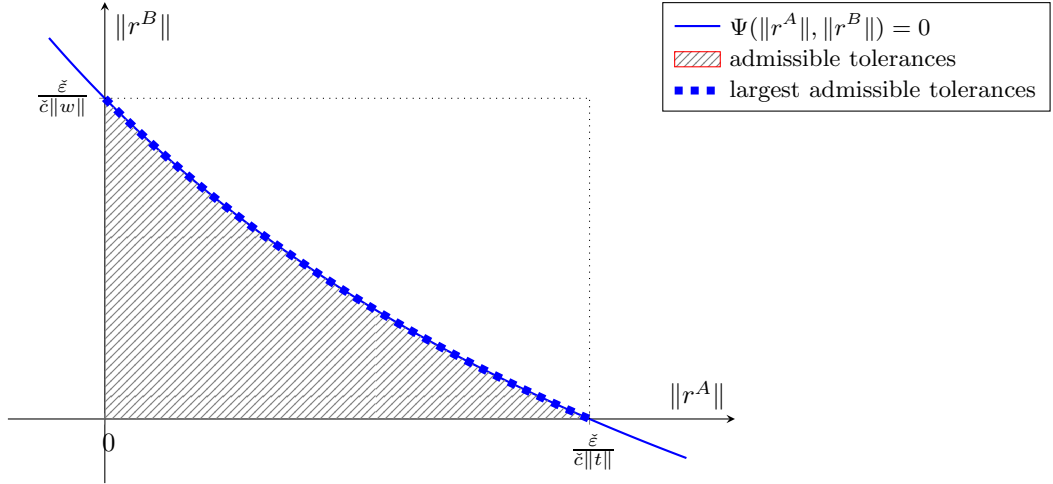
By resolving $\Psi=0$, e.g., for $\|r_k^B\|$:
\begin{align}\label{tolBfromtolA}
\|r_k^B\|=\frac{\check{\varepsilon}_k-\check{c}_k\|r_k^A\|\|t_{k-1}\|}{\check{c}_k(2\|r_k^A\|+\|w_{k-1}\|)},
\end{align}
we can infer an upper bound for one inner residual from the other. Clearly, while obeying the constraints~\eqref{relax_AB}, the smaller $\|r_k^A\|$ is chosen, the larger $\|r_k^B\|$ can be and vice versa.

If the computed Sylvester residual decreases in norm, the product of the norms of the residual factors $\|w_{k-1}\|\|t_{k-1}\|$ will decrease as well but not necessarily the norms of the individual residual factors $\|w_{k-1}\|$, $\|t_{k-1}\|$. Hence, the absolute inner residual norms do not need to increase as the outer residual norm $\|\cR^{\text{comp.}}_k\|$ decreases which forms another difference to the Lyapunov situation.
However, we observe increasing relative inner residual norms $\frac{\|r_k^A\|}{\|w_{k-1}\|}\leq\frac{\check{\varepsilon}_k}{\check{c}_k\|w_{k-1}\|\|t_{k-1}\|}$, $\frac{\|r_k^B\|}{\|t_{k-1}\|}\leq\frac{\check{\varepsilon}_k}{\check{c}_k\|w_{k-1}\|\|t_{k-1}\|}$ since $\check{c}_k$ can be bounded by a moderate constant as we will see later.
\subsection{Distance between exact and inexact Sylvester ADI with the dynamic inner solve tolerances}
We now establish a relation between the (norms of the) computed Sylvester residual matrices of both exact and inexact Sylvester-ADI, where the dynamic inner stopping criteria from Theorem~\ref{thm:ADI_theorelax} is used in the latter. 

\begin{theorem}\label{thm:ADI_resdecay}
Assume $\kmax$ steps of both exact and inexact Sylvester-ADI are applied to a Sylvester equation using the same set of shift parameters $(\alpha_k,~\beta_k)$, $j=1,\ldots,\kmax$.
Let quantities of the exact LR-ADI iteration be denoted by superscripts $^{\text{exact}}$ and let $0<\varepsilon\ll 1$ be a small threshold.
If the condition~\eqref{ADI_relax1_theo} is used in the inexact ADI iteration, then 
\begin{align*}
\|\cR^{\text{comp.}}_{\kmax}\|\leq\|\cR^{\text{exact}}_{\kmax}\|+\cO(\varepsilon).
\end{align*}   
\end{theorem} 
\begin{proof}
At first, following the construction of the residual factors $w_k,~t_k$ in Algorithm~\ref{alg:lradi}, we find that
\begin{align*}
w_k&=\cC_k(AM^{-1})w_{k-1}+(I-\cC_k(AM^{-1}))r_k^A\\
&=\ldots=w_k^{\text{exact}}+\sum\limits_{j=1}^k\left[\prod\limits_{i=j+1}^k\cC_i(AM^{-1})\right](I-\cC_j(AM^{-1}))r_j^A\\
t_k&=\cC_k(C^{-1}B)^*t_{k-1}+(I-\cC_k(C^{-1}B)^*)r_k^B\\
&=\ldots=t_k^{\text{exact}}+\sum\limits_{j=1}^k\left[\prod\limits_{i=j+1}^k\cC_i(C^{-1}B)^*\right](I-\cC_j(C^{-1}B)^*)r_j^B
\end{align*}   
which generalizes \cite[Lemma 3.1]{KueF20}.  
Then,
\begin{align*}
\|\cR^{\text{comp.}}_{\kmax}\|=&\|w_kt_k^\trp\|\\
\leq&\|(\cC_k(AM^{-1}))w_{k-1}t_{k-1}^\trp \cC_k(C^{-1}B)\|+\|(I-\cC_k(AM^{-1}))r_k^At_{k-1}^\trp \cC_k(C^{-1}B)\|\\
&+\|((\cC_k(AM^{-1})w_{k-1}(r_k^B)^\trp(I-\cC_k(C^{-1}B))\|\\
&+\|(I-\cC_k(AM^{-1}))r_k^A(r_k^B)^\trp(I-\cC_k(C^{-1}B))\|\\
\leq& \|(\cC_k(AM^{-1}))w_{k-1}t_{k-1}^\trp \cC_k(C^{-1}B)\|\\
&+(c_k^A+1)c_k^B\|r_k^A\|\|t_{k-1}\|+(c_k^B+1)c_k^A\|r_k^B\|\|w_{k-1}\|+(c_k^B+1)(c_k^A+1)\|r_k^A\|\|r_k^B\|,
\end{align*}
where we used the constants $c_k^A,~c_k^B$ introduced in Theorem~\ref{thm:ADI_theorelax}. Consequently, 
\begin{align*}
&(c_k^A+1)c_k^B\|r_k^A\|\|t_{k-1}\|+(c_k^B+1)c_k^A\|r_k^B\|\|w_{k-1}\|+(c_k^B+1)(c_k^A+1)\|r_k^A\|\|r_k^B\|\\
&\leq (c_k^A+1)(c_k^B+1)\left(\|r_k^A\|\|t_{k-1}\|+\|r_k^B\|\|w_{k-1}\|+2\|r_k^A\|\|r_k^B\|\right)\\
&\leq \check{c}_k^2\left(\|r_k^A\|\|t_{k-1}\|+\|r_k^B\|\|w_{k-1}\|+2\|r_k^A\|\|r_k^B\|\right)\leq \check{c}_k\frac{\varepsilon}{\kmax}.
\end{align*}
Repeating this process again for $w_{k-1}$, $t_{k-1}$ for $1\leq k\leq \kmax-1$ eventually yields
\begin{align*}
\|\cR^{\text{comp.}}_{\kmax}\|&=\|\left(\prod\limits_{k=1}^{\kmax}\cC_k(AM^{-1})\right)w_{0}t_0^\trp\left(\prod\limits_{k=1}^{\kmax}\cC_k(C^{-1}B)\right)\|^2+\sum\limits_{j=1}^{\kmax}\check{c}_k\frac{\varepsilon}{\kmax}\\
&=\|\cR^{\text{exact}}_{\kmax}\|+\frac{\varepsilon}{\kmax}\sum\limits_{k=1}^{\kmax}\check{c}_k.
\end{align*}
\end{proof}

Combining Theorems~\ref{thm:ADI_theorelax} and~\ref{thm:ADI_resdecay} yields the following conclusion.
\begin{corollary}\label{cor:trueres}
Under the same conditions as Theorem~\ref{thm:ADI_resdecay} we have 
\[
\|\cR^{\text{true}}_{\kmax}\|\leq\|\cR^{\text{comp}}_{\kmax}\|+\|\Delta\cR_{\kmax}\|\leq \|\cR^{\text{exact}}_{\kmax}\|+(1+\sum\limits_{k=1}^{\kmax}\frac{\check{c}_k}{\kmax})\varepsilon.
\]
\end{corollary}

Hence, if~\eqref{ADI_relax1_theo} is used, then the (true) Sylvester residual norms in the inexact LR-ADI are a small perturbation of the
residuals of the exact method, provided the $\check{c}_k$ are bounded by moderate constants, which we will discuss next.

\subsection{Practical and Implementational Considerations}
Here we discuss some ways for the pratictal usage of the stopping criteria~\eqref{ADI_relax1_theo} in an actual implementation of the low-rank Sylvester ADI.
\subsubsection{Estimating the spectral norms}
At first we are going to bound the constants $c_k^A,~c_k^B$ and $\check{c_k}$. 
For this we insert some additional assumptions:
\begin{itemize}
\item The rational functions defining both Cayley transforms, 
	\[
	\phi_k^A(x):=\frac{x-\alpha_k}{x+\beta_k},\quad \phi_k^B(x):=\frac{x-\beta_k}{x+\alpha_k},\quad k\geq 1,
	\]
	are $\forall k\geq 1$ analytic on $\cW(AM^{-1})$ and, respectively, $\cW(C^{-1}B)$. 

	\item It holds
	\[
	q^A:=\max\limits_{z\in\cW(AM^{-1})}\vert\phi_k^A(z)\vert<1,\quad q^B:=\max\limits_{z\in\cW(C^{-1}B)}\vert\phi_k^B(z)\vert<1.
\]
\end{itemize}
Then by~\cite{CroP17},
\[
c_k^A=\|\cC_k(AM^{-1})\|\leq \psi^A q^A<\psi^A,\quad c^B_k=\|\cC_k(C^{-1}B)\|\leq  \psi^B q^B<\psi^B,
\]
where $\psi^A=1+\sqrt{2}$ ($\psi^A=1$ if $AM^{-1}$ is normal) and similarly for $\psi^B$ and $C^{-1}B$.
As a consequence, we can bound the constants as
\[
\check{c}_k\leq c:=2+\sqrt{2}.
\]
\paragraph{Practical stopping criteria}
Additionally, if we use $\dfrac{\varepsilon}{2\check c_k\kmax}$ instead of $\dfrac{\varepsilon}{k_{\max}}$ for the condition in Theorem~\ref{thm:ADI_theorelax}, than we achieve $\|\cR^{\text{comp}}_{\kmax}\|\leq\|\cR^{\text{exact}}_{\kmax}\|+\frac{\varepsilon}{2}$ and  $\|\cR^{\text{true}}_{\kmax}\|\leq \|\cR^{\text{exact}}_{\kmax}\|+\varepsilon$ in Corollary~\ref{cor:trueres}. 
Replacing $\check{c}_k$ by the bound $c$ leads to
\begin{align}\label{ADI_relax1_prak}
	\|r^A_k\|\|t_{k-1}\|+\|r^B_k\|\|w_{k-1}\|+2\|r^B_k\|\|r^A_k\| \leq \xi\frac{\varepsilon}{2c^2\kmax}=:\hat\varepsilon_k^{(1)}
\end{align} 
as practical realization of~\eqref{ADI_relax1_theo}. Here, $0<\xi\leq 1$ is a save-guard constant for situations when the above assumptions are midly violated.
Note that similar small save-guard constants are common in inexact (rational) Krylov methods~\cite{SimEld2002,Sim05,KueF20}. In most of our experiments, $\xi=1$ was sufficient.
\subsubsection{Incorporation of the inner residual norms and residual gaps}
The proposed stopping criteria require the norms residuals of the shifted linear systems which can often be directly obtained from the employed preconditioned Krylov subspace methods. This holds especially in the case of of right preconditioning, if left or two-sided preconditioning is used, they might have to be computed directly or estimated differently. 

\paragraph{Back-Looking}
The previous inner residual norms can be used to further refine the dynamic stopping strategy. In the bound~\eqref{ADIresgapbound1} we can at step $k$ try to achieve
\begin{align*}
	\|\Delta\cR_{k}\|&\leq\|\Delta\cR_{k-1}\|+\|r^B_k\|\|\gamma_kMz_k\|+\|r^A_k\|\|\overline{\gamma_k}C^*y_k\|\leq \frac{k\varepsilon}{\kmax}.
\end{align*}
which leads to
\begin{align}\label{ADI_relax2_theo} 
\check{c}_k(\|r^A_k\|\|t_{k-1}\|+\|r^B_k\|\|w_{k-1}\|+2\|r^B_k\|\|r^A_k\|)+\|\Delta\cR_{k-1}\|\leq \frac{k\varepsilon}{\kmax}
\end{align}
As in~\cite{KueF20}, the reasoning behind this approach is to look back at all previous inner residuals and also incorporates the previous residual gap $\|\Delta\cR_{k-1}\|$. Hence, we will refer to this strategy as 'back-looking'. This might allow to use at step $k$ slightly larger inner tolerances if, at some of the earlier steps $i\leq k-1$, smaller inner residuals than requested were achieved by the inner solvers.

The back-looking strategy~\eqref{ADI_relax2_theo} requires the previous residual gap $\|\Delta\cR_{k-1}\|=\|\eta^A_{k-1}+\eta^B_{k-1}\|$ which might be expensive to compute and, moreover, it would require storing all previous inner residuals. 
Here, we simply use for $k\geq 2$ the approximations
\begin{align}\label{upd_resgap}
\begin{split}
\|\Delta\cR_{k-1}\|&\leq\|\eta^A_{k-1}\|+\|\eta^B_{k-1}\|\leq u_{k-1}+v_{k-1},\\
u_{k-1}&:=u_{k-2}+|\gamma_{k-1}|\|Mz_{k-1}\|\|r^B_{k-1}\|,\quad u_0:=0,\\
v_{k-1}&:=v_{k-2}+|\gamma_{k-1}|\|C^\trp y_{k-1}\|\|r^A_{k-1}\|,\quad v_0:=0.
\end{split}
\end{align} 
The matrix-vector products $Mz_{k-1}$ and $C^\trp y_{k-1}$ can be reused from step~\ref{lradi_resup} of Algorithm~\ref{alg:lradi}.

As a practical realization of~\eqref{ADI_relax2_theo} we propose 
\begin{align}\label{ADI_relax2_prak}
	\|r^A_k\|\|t_{k-1}\|+\|r^B_k\|\|w_{k-1}\|+2\|r^B_k\|\|r^A_k\| \leq \frac{1}{c}\left|\xi\frac{k\varepsilon}{2c\kmax}-u_{k-1}-v_{k-1}\right|=:\hat\varepsilon_k^{(2)}
\end{align} 
(Here, we again replaced $\check{c}_k$ by $c$, divided the right hand side of~\eqref{ADI_relax2_theo}  by $2c$, and introduced the save-guard constant $\xi$).
\subsubsection{Selecting one particular combination of solve tolerances}
From the infinitely many possibilities for the inner tolerances, we have to select one combination in an actual implementation.
At first, for reasons of feasibility, we may restrict the inner accuracies to some minimal and maximal levels via
\begin{align}\label{bound_tols}
0<\delta^A_{\min}\leq \delta_k^A\leq \delta^A_{\max},\quad 0<\delta^B_{\min}\leq \delta_k^B\leq \delta^B_{\max}.
\end{align}
Then, one simple selection for $\delta_k^A$ could be to pick it somewhere from the middle of the boundary curve of~\eqref{relax_AB} and compute $\delta_k^B$ via~\eqref{tolBfromtolA}, as illustrated in Figure~\ref{fig:res_admiss2}. A strategy that worked well in our experiments is to set 
\begin{align}\label{simple_tols}
\begin{split}
\delta_k^A&=\max\left(\half\left(\min(\delta_{\max}^A,\frac{\hat\varepsilon_k}{\|t_{k-1}\|})-\delta_{\min}^A\right),\delta_{\min}^A\right)\\
\delta_k^B&=\max\left(\min\left(\frac{\hat\varepsilon_k-\delta_k^A\|t_{k-1}\|}{2\delta_k^A+\|w_{k-1}\|},\delta_{\max}^B\right),\delta_{\min}^B\right),
\end{split}
\end{align}
where $\hat\varepsilon_k$ can be $\hat\varepsilon_k^{(1)}$ or $\hat\varepsilon_k^{(2)}$, depending on whether~\eqref{ADI_relax1_prak} or \eqref{ADI_relax2_prak} is used.
\begin{figure}
\begin{tikzpicture}[/pgf/declare function={f=(1-x)/(2/3*0.8*x+1);},/pgf/declare function={g=0.2;}]

\begin{axis}[%
width=0.7\linewidth, 
height=0.5\linewidth,
axis lines=middle,
xmin=-0.2,
xmax=1.3,
ymin=-0.2,
ymax=1.3,
xlabel={$\|r^A\|$},
ylabel={$\|r^B\|$},
restrict y to domain=-0.2:1.2,
xtick={1},
ytick={0,1},
xticklabels={$\frac{\check{\varepsilon}}{\check{c}\|t\|}$},
yticklabels={0,$\frac{\check{\varepsilon}}{\check{c}\|w\|}$},
extra x ticks={0},
extra x tick labels={0},
x tick label style={xshift={0.15em}},
axis background/.style={fill=white},
legend style={draw=black,fill=white,font=\small,legend cell align=left,legend columns=1,anchor=north west,at={(0.6,1.0)}},
legend entries={{$\Psi(\|r^A\|,\|r^B\|)=0$}, admissible tolerances, admissible tolerances on $\Psi=0$}
]
\addplot [color=blue,thick,domain=-0.2:1.2,samples=100, name path=A]
   {f};

\addplot[dashdotted, thick,red,forget plot,name path=B,domain=0:1,samples=2] {g}; 

\addplot[mark=none,draw=none,area legend,pattern=north east lines,pattern color=black!50] fill between[of=A and B,soft clip={domain=0.1:0.7229}];
	%
\addplot [color=blue,dashed,line width=3pt,domain=0.1:0.7229,samples=100,opacity=0.5]
   {f};	
	\label{legend1}
	
\addplot[dotted, thin,black] coordinates {(1,0) (1,1)};
\addplot[dotted, thin,black] coordinates {(0,1) (1,1)};		
\addplot[dashdotted, thick,red] coordinates {(0.1,0) (0.1,1)};

\addplot[mark=square*,only marks,mark size=3.0pt,fill=yellow, mark options={solid},forget plot] coordinates {(0.4,0.49)};
	%
\end{axis}

\end{tikzpicture}%
\caption{Extension of Figure~~\ref{fig:res_admiss} illustrating the region of admissible inner residual norms with set of admissible combinations of largest tolerances (thick dashed line). The square marks on possible combination from this set. The straight dashed-dotted lines indicate minimal bounds $\delta^A_{\min},~\delta^B_{\min}$ for the inner residuals norms.}
\label{fig:res_admiss2}
\end{figure}
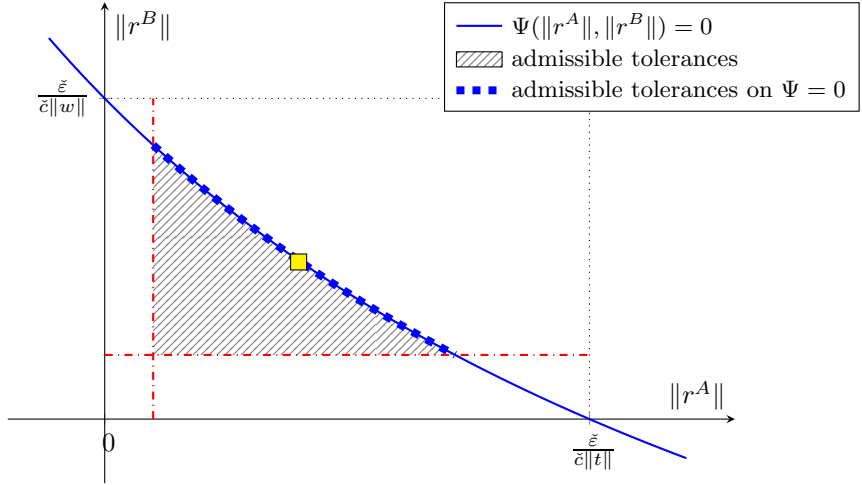

Since any point of the curve $\Psi=0$ is admissible, we may move this point so that smaller inner residuals are favoured for those linear system that are easier and/or less costly to solve and, thus, allow larger residuals for the other linear system. For example, if the linear system with $A+\beta_k M$ cab ve easier and faster solved than that with $B+\alpha_k C$, we set
\begin{align}\label{one_side}
\delta_k^A=\delta^A_{\min}\quad\text{and}\quad \delta_k^B\quad\text{via}\quad\eqref{simple_tols}.
\end{align}
 In the reverse situation, analogously select $\delta_k^B=\delta^B_{\min}$ and largest admissible $\delta_k^A$. 
As indicators on how difficult or costly the iterative solution of a linear system is, we could look at, e.g., the size of the matrix, the cost of the matrix vector products (sparsity density), the condition number of the matrix, or a mixture thereof.  
\paragraph{Combination of direct and iterative linear solves}
A similar situation arises when the linear systems of one sequence, e.g. those with $A+\beta_k M$, can be solved by sparse direct solvers, and the other linear systems, e.g. those with $B+\alpha_k C$, require iterative methods, or vice-versa. Direct solvers usually achieve a high accuracy and do not need stopping criteria, so that selecting a residual threshold is only required in the other sequence. For example, if the systems with $A+\beta_k M$ are solved directly and we, ideally, have $\|r^A\|\approx 0$, than the proposed practical stopping criteria~\eqref{ADI_relax1_prak},\eqref{ADI_relax2_prak} simplify to
	\[
\|r_k^B\|\lesssim \frac{\hat\varepsilon_k}{\|w_{k-1}\|}.
	\]
Note that this bears similarity with the simplified stopping criteria for the inexact LR-ADI iteration for Lyapunov equations~\cite{KueF20}. There, the bounds are  of the form $\|r_k\|\lesssim\frac{\hat\varepsilon_k}{\sqrt{\|\cR^{\text{comp}}_{k-1}\|}}$, where $\|w_{k-1}\|=\|\cR^{\text{comp}}_{k-1}\|^\half=$. For the Sylvester ADI and $r=1$, the computed Sylvester residual norms are $\|\cR^{\text{comp}}_{k-1}\|=\|w_{k-1}\|\|t_{k-1}\|$, so that $\|w_{k-1}\|=\|\cR^{\text{comp}}_{k-1}\|^p$ for some $0<p<1$.
\subsubsection{Further implementation aspects}
\paragraph{Choice of the inner iterative solver and preconditioning}
The motivation behind low-rank solvers for large matrix equations was to compute approximate solution in a memory efficient way. This should be maintained also in inexact low-rank method and we, therefore, strongly advocate the use of short-recurrence Krylov methods for solving the (unsymmetric) inner linear systems.
Working with an (non-restarted) long-recurrence methods such as GMRES for solving unsymmetric inner linear systems requires storing the full Krylov basis and, hence, might hinder working in a memory efficient way (especially if the memory requirements for GMRES exceed those for the low-rank factors). Of course, having a very  effective preconditioner
would limit the memory requirements.

If $\rank{fg^\trp}=r>1$, every inner linear system has $r$ right hand sides. One could then use special block Krylov methods (see, e.g.,~\cite{OLEARY80,Sood15}). In our experiment this had no advantage over simply employing the single vector methods to every column  
$w_{k-1}(:,\ell)$,  $t_{k-1}(:,\ell)$, $\ell=1,\ldots,r$. As stopping criteria we than simply used $\|r_k(:,\ell)\|\leq\delta_k/r$.

The proposed stopping criteria require the norms of the residuals $r_k$ of the underlying inner linear systems. Naturally, one would use preconditioners within the Krylov methods to improve their performance. If left or two-sided preconditioning is used, the preconditioned Krylov method will internally work with the preconditioned residuals which might be different from the true inner residuals. Hence, we therefore used mostly right preconditioning which does not suffer from this issue. In the other cases we computed the true inner residual norms after termination of the Krylov method and ensured than $\|r_k\|\leq \delta_k$ was met. 
\paragraph{Complex Shift Parameters}
For Sylvester equations defined by real but non-symmetric coefficient matrices, the sets of shifts $\lbrace\alpha_j\rbrace_{j=1}^k,~\lbrace\beta_j\rbrace_{j=1}^k$ can 
include pairs of complex conjugated shifts. In order to minimize the amount of complex arithmetic operations and to generate real low-rank solution factors $Z_k,\Gamma_k,Y_k$, 
the results in~\cite{BenK14,Kue16} can be used. The main idea is to perform double iteration step when a complex pair of $\alpha$-shifts meets a complex pair of $\beta$-shifts or two real $\beta$-shifts or vice-versa (all other situations can be avoided by a basic reordering of the shift sequences). This way, only one complex linear system needs to be solved per complex conjugate pair of shifts, but the arising variant of the low-rank Sylvester ADI involves rather cumbersome formulas. Therefore, we stick here to the possible complex formulation of the method. Note that in most of the upcoming experiments the used shifts were entirely real.  
As in the inexact low-rank Lyapunov ADI iteration, using the real formulation of the Sylvester ADI only involves some minor adjustments for the estimate~\eqref{upd_resgap} when the back-looking strategy~\eqref{ADI_relax2_prak} is used and a double-step occurs.
\paragraph{Related Matrix Equations}
Next to generalized Lyapunov equation ($B=A^\trp,~C=M^\trp,~g=f$) also cross-Gramian Sylvester equations ($B=A,~C=M$) and symmetric Stein matrix equations ($B=M^\trp$, $C=-A^\trp$, $g=f$) are special cases of~\eqref{gsylv}. Hence, with the appropriate adaptations Algorithm~\ref{alg:lradi} can be employed as well (see~\cite{BenK14, Kue16}) including the proposed dynamic stopping criteria of this work. 
\section{Numerical experiments}\label{sec:num}
The following experiments were carried out in \matlab~2023a on a \intel\coretwo~i7-7500U CPU @ 2.7GHz with 16 GB RAM\footnote{The codes for the experiments will be made available online after publication of the article.}.
We wish to obtain an approximate solution such that the scaled Sylvester residual norm satisfies
\begin{align*}
 \fR:=\|\cR^{\text{true}}\|/\|fg^\trp\|\leq \tilde\varepsilon,\quad 0<\tilde\varepsilon\ll 1.
\end{align*}
In all the upcoming experiments $\tilde\varepsilon=10^{-8}$ is used.

As Krylov subspace solvers for the inner linear systems we use BiCGstab for unsymmetric coefficients $A+\beta_kM$, $B+\alpha_k C$ and MINRES in the symmetric case. 
Note that if $\beta_k\in\C$ then $A+\beta_kM$ is unsymmetric even if $A,~M$ are symmetric and, thus, requiring an Krylov method for unsymmetric linear systems.
Sparse-direct solves are carried out by the \matlab~\textit{backslash}-routine. 

Shift parameters for the Sylvester-ADI iteration are either generated by the heuristic approach from~\cite{BenLT09} (using $10$  Ritz values for each of $A,~B$ and $20$ inverse Ritz values of $A^{-1},~B^{-1}$) or the analytic approach by Sabino~\cite[Algorithm~2.1]{Sab07} (using approximations of the extremal eigenvalues from both spectra obtained with the \matlab routine \texttt{eigs}). 

The coefficients in some of our experiments come from finite-difference discretizations of convection-diffusion operators 
\begin{align}\label{ex:convdiff}
\cL(u)=-\triangle u+\omega^\trp\nabla u\quad\text{on}\quad (0,~1)^3,
\end{align}
with different $\omega\in\R^3$ and a uniform grid with $n_0$ points in each spatial dimension were used, leading to matrices of size $n_0^3\times n_0^3$. The right hand side factors $f,g$ are drawn from a normal distribution and are rescaled so that $\|f\|=\|g\|$. Table~\ref{tab:examples} gives an overview of the examples and the settings for preconditioners and ADI shift selection methods.

\begin{table}[t]
  \centering
  \caption{Examples used in the experiments with matrix properties, settings for preconditioners and shift selection.
Here, iLU$(X,\nu)$ and iC$(X,\nu)$ refer to incomplete LU and, respectively, Cholesky factorization of the matrix $X$ with drop
tolerance $\nu$. The last column indicates which ADI shift selection strategy (analytic approach by Sabino or heuristic method) is used. }
\setlength{\tabcolsep}{0.5em}
\renewcommand{\arraystretch}{1.2}
  \begin{tabularx}{\textwidth}{l|l|X|l|l|l|l}
    Ex. & $n$, $m$&coefficients&$r$&sym.&prec.&shifts\\
    \hline
    1&125000&$A$: $\omega=0$, $M=I_n$&5&yes&iC($-A$,0.1)&Sab.\\
		&27000&$B$: $\omega=0$, $C=I_m$&5&yes&iC($-B$,0.1)&Sab.\\
\hline
    2&500000&$A$: $\omega=\smb x\sin(x),y\cos(y),\mathrm{e}^{z^2-1}\sme^T$, $M=I_n$&3&no&iLU($A$,0.1)&heur.\\
		&27000&$B$: $\omega=\smb zy(x^2-1),\frac{1}{y^2+1},\mathrm{e}^z\sme^T$, $C=I_m$&3&no&iLU($B^\trp$,0.1)&heur.\\
    \hline
	3&125000&$A$ from Example 1&5&yes&iC($-A$,0.1)&heur.\\
		&22500&$B$: two-dim. version of \eqref{ex:convdiff} with $\omega=0$, $C=I_m$&5&yes&--&heur.\\
    \hline
	4&106641&$A,~M$ from simplifiedMachineToolFineSA1~\cite{dataSauNVetal23}&2&yes&iC($-A-\beta_kM$,0.1)&heur.\\
		&35408&$B,~C$ from simplifiedMachineToolFineSA2~\cite{dataSauNVetal23}&2&yes&iC($-B-\alpha_kC$,0.1)&heur.\\
  \end{tabularx}\label{tab:examples}
\end{table}

Example~2 is set up similar to an example in~\cite{KresLMP21}. The matrix $B$ in Example~3 comes form a two-dimensional finite difference discretization, and the arising linear systems with can be efficiently solved by direct methods. Hence, only the inner tolerances $\delta_k^A$ need to be chosen. The matrices of the generalized Sylvester equation in Example~4 come from a finite element discretization of the heat transfer across a machine tool~\cite{SauVNetal20}. Here, updating the preconditioners in every ADI iteration step is required to the achieve a reasonable performance of the inner solver MINRES. In all other experiments fixed preconditioners were sufficient.

We will monitor the performance of Sylvester-ADI with fixed tolerances, using $\delta_k^{A,B}=\tilde\varepsilon/20$, and with the proposed dynamically chosen inner solve tolerances. If not stated otherwise, $\delta_{\min}=\tilde\varepsilon/20$, $\delta_{\max}=0.1$ are set as minimal and maximal linear solve tolerances and 
 $j_{\max}=50$, $\xi=1$ are used. Only example~4 required slightly stricter settings: $\delta_{\min}=\tilde\varepsilon/100$, $j_{\max}=100$, $\xi=0.1$.

The following settings for the dynamic stopping criteria are tested:
\begin{itemize}
	\item dynamic, no BL, mid: strategy~\eqref{ADI_relax1_prak} without back-looking, combination~\eqref{simple_tols} of $\delta_k^A,\delta_k^B$ from middle of the admissible set.
	\item dynamic, BL, mid: like above, but with back-looking~\eqref{ADI_relax2_prak}.
		\item dynamic, no BL, $B$: strategy~\eqref{ADI_relax1_prak} without back-looking, inner iterations with smaller systems (defined by $B^\trp+\overline{\alpha} C^\trp$) are preferred via \eqref{one_side}. 
	\item dynamic, BL, $B$: like above, but with back-looking~\eqref{ADI_relax2_prak}.
\end{itemize}
Inside Algorithm~\ref{alg:lradi}, the scaled computed Sylvester residual norm $\|\cR^{\text{comp}}\|/\|fg^\trp\|$ is used for the outer stopping criteria. After termination, we also estimate the (scaled) norm of the true Sylvester residual matrix $\|\cR^{\text{true}}\|$ by using the \texttt{eigs} routine to get an approximation of $\lambda_{\max}\left((\cR^{\text{true}})^\trp\cR^{\text{true}}\right)$. The results are summarized in Table~\ref{tab:ex_results}. 

\begin{table}[t]
  \centering
  \caption{Experimental results. The columns denote the used inner stopping criterion (fixed or dynamic versions), the
number of required outer iterations it\textsuperscript{out}, the column dimension of the low-rank solution factors (dim), the final obtain scaled residual norm $\fR_{k}$, the total number of inner iteration steps $\sum$it$^{\text{in},A}$, $\sum$it$^{\text{in},B}$ for both linear systems, the computing times in seconds, and 
the obtained savings in the computing time compared to inexact LR-ADI with fixed inner tolerances. 
}\label{tab:ex_results}
  \setlength{\tabcolsep}{0.5em}
  \begin{tabularx}{\textwidth}{r|X|r|r|r|rr|r|r}
Ex.&settings inner tol&it\textsuperscript{out}&dim&$\fR_{k}$&$\sum$it$^{\text{in},A}$&$\sum$it$^{\text{in},B}$&time&save\\
\hline
\multirow{6}{*}{%
    \begin{minipage}{1em}
		1
    \end{minipage}
    }   	
&direct solves&29&145&1.9e-09&--&--&114.6&\\ 
&fixed&29&145&1.9e-09&581&738&34.3&\\ 
&dynamic, no BL, mid&29&145&1.9e-09&471&462&26.3&23.32\%\\ 
&dynamic, BL, mid&29&145&2.0e-09&436&437&25.8&24.78\%\\ 
&dynamic, no BL, $B$&29&145&1.9e-09&450&597&27.9&18.66\%\\ 
&dynamic, BL, $B$&29&145&1.9e-09&424&597&28.2&17.78\%\\ 
\hline
\multirow{3}{*}{%
    \begin{minipage}{1em}
		2
    \end{minipage}
    }   	
&fixed&25&75&4.1e-10&836&428&234.6&\\ 
&dynamic, BL, mid&25&75&4.1e-10&602&315&165.4&29.92\%\\ 
&dynamic, BL, $B$&25&75&4.2e-10&596&350&162.8&30.61\%\\ 
\hline
\multirow{3}{*}{%
    \begin{minipage}{1em}
		3
    \end{minipage}
    }   	
&direct solves&20&100&5.3e-09&--&--&77.2&\\ 
&fixed&20&100&5.3e-09&638&--&24.2&\\ 
&dynamic, BL&20&100&5.3e-09&430&--&18.0&25.62\%\\
\hline
\multirow{4}{*}{%
    \begin{minipage}{1em}
		4
    \end{minipage}
    }   	
&direct solves&70&140&6.4e-09&--&--&184.5&\\ 
&fixed&70&140&6.4e-09&3330&2827&194.4&\\ 
&dynamic, BL, mid&70&140&7.0e-09&1524&1617&109.3&43.78\%\\ 
&dynamic, BL, $B$&75&150&9.1e-09&1491&2388&129.5&33.38\%\\ 
\end{tabularx}
\end{table}

For Example~1 we tested the inexact LR-Sylvester-ADI iteration with all of the above dynamic stopping criteria as well as with fixed inner tolerances and also used the exact ADI (with direct inner solvers). From the data collected in Table~\ref{tab:ex_results} we observe, at first, that with an appropriate choice for the 
inner tolerances, the inexact Sylvester-ADI needs the same number of outer steps and achieves similar final Sylvester residuals than the exact counterpart, but requires significantly less computing time. This is also the case for most of the other examples. Secondly, using the dynamic stopping criteria leads overall to smaller numbers of the required inner iteration steps (column $\sum$it$^{\text{in},A}$, $\sum$it$^{\text{in},B}$ in the table) compared to fixed inner tolerances. This leads to reduced computing times with savings between approximately 17 and 25 percent for Example~1.

Now comparing the various different versions of the dynamic stopping criteria, we observe that the plain version~\eqref{ADI_relax1_prak} requires slightly more inner iterations than the version with back-looking~\eqref{ADI_relax2_prak}. Since using~\eqref{ADI_relax2_prak} comes with almost no additional costs, we therefore always use back-looking for the remaining examples. The strategies which aim at more inner iteration steps with the smaller linear systems indeed achieve this goal: the numbers $\sum$it$^{\text{in},A}$ are slightly decreased while $\sum$it$^{\text{in},B}$ are slightly increased. However, this does for Example~1 not lead to a reduction in the computing time.. Some more fine-tuning regarding the selection of a combination $\delta_k^A,~\delta_k^B$ might be needed here. Figure~\ref{fig:ex1_res} shows that, when the scaled computed Sylvester residual norm $\fR_k^{\text{comp.}}$ decreases in the course of the Sylvester-ADI iteration, the dynamic stopping criteria lead to increasing inner residual norms $\|r_k^A\|$, $\|r_k^B\|$. Note that since the  $\fR_k^{\text{comp.}}$ was visually indistinguishable for all the tested variants, only one curve is shown in~Figure~\ref{fig:ex1_res}. In Figure~\ref{fig:ex1_innit} the cumulative sum of the inner iteration steps is illustrated for different inner stopping criteria. We clearly see that the dynamic criteria lead to a significantly reduced slope compared to fixed tolerances, which is reduced slightly further when back-looking~\eqref{ADI_relax2_prak} is equipped.  

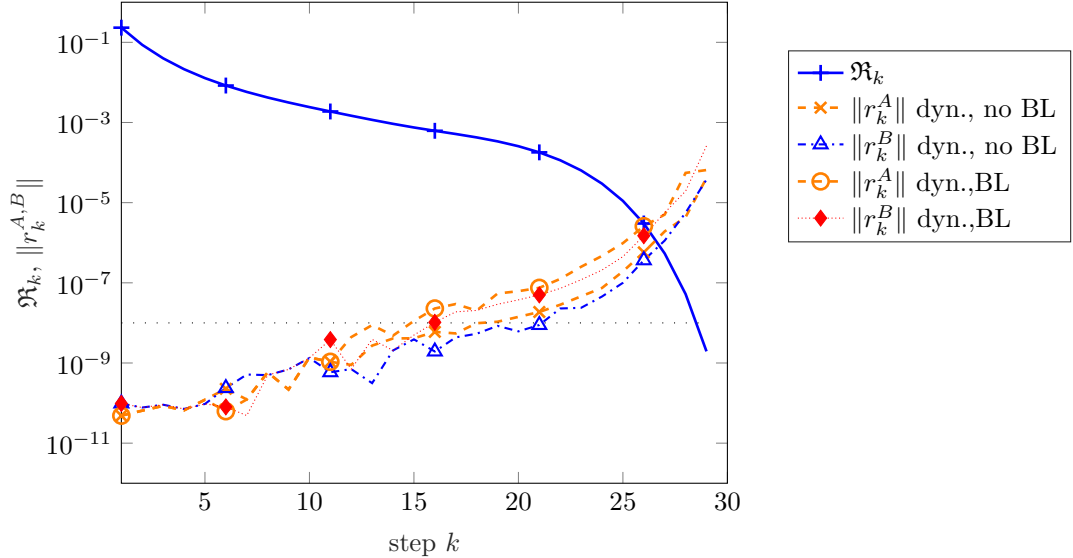
\begin{figure}[t]
%
%
\definecolor{mycolor1}{rgb}{1.00000,0.00000,1.00000}%
\begin{tikzpicture}

\begin{axis}[%
width=0.6\linewidth,
height=0.5\linewidth,
cycle list name=res,
xmin=1,
xmax=30,
xlabel style={font=\color{white!15!black}},
xlabel={step $k$},
ylabel={$\fR_k$, $\|r_k^{A,B}\|$},
ymode=log,
ymin=1e-12,
ymax=1,
yminorticks=true,
ylabel style={font=\color{white!15!black}},
axis background/.style={fill=white},
legend style={at={(1.1,0.9)}, anchor=north west, legend cell align=left, align=left, draw=white!15!black},
legend entries= {$\fR_k$, {$\|r_k^A\|$ dyn., no BL}, {$\|r_k^B\|$ dyn., no BL},{$\|r_k^A\|$ dyn.,BL},{$\|r_k^B\|$ dyn.,BL}}
]
\addplot 
  table[row sep=crcr]{%
1	0.231894311509842\\
2	0.0863407980086386\\
3	0.0398920280742611\\
4	0.0213504919440937\\
5	0.0127984607715732\\
6	0.00834948611762347\\
7	0.00578839797341375\\
8	0.00420043170587674\\
9	0.00315274824933197\\
10	0.00241996400099272\\
11	0.00188212263318951\\
12	0.0014746277118499\\
13	0.00116241038573174\\
14	0.000925697858979264\\
15	0.000750786918289523\\
16	0.000620520006377795\\
17	0.000516163162677393\\
18	0.000424397043884384\\
19	0.00033818668291699\\
20	0.000255882639188282\\
21	0.000179790295032347\\
22	0.000114297144000769\\
23	6.36523549256432e-05\\
24	2.97790026967087e-05\\
25	1.10644049680798e-05\\
26	3.01956979945373e-06\\
27	5.39602513695608e-07\\
28	5.2414288304444e-08\\
29	1.95276558227478e-09\\
};

\addplot 
  table[row sep=crcr]{%
1	4.85051246011334e-11\\
2	6.47384112875604e-11\\
3	8.53748642691232e-11\\
4	6.59101662796204e-11\\
5	1.21218687079376e-10\\
6	2.27155808617969e-10\\
7	1.23060871459558e-10\\
8	5.85894322064088e-10\\
9	2.179346954662e-10\\
10	1.37531307798068e-09\\
11	1.07165689890225e-09\\
12	8.94885969705691e-10\\
13	2.71158188894126e-09\\
14	4.09185800421101e-09\\
15	4.13987510388844e-09\\
16	5.98221793169549e-09\\
17	5.46369566661057e-09\\
18	9.82501572664878e-09\\
19	1.06882630694862e-08\\
20	1.41511089697922e-08\\
21	1.8735037711794e-08\\
22	2.83749100206088e-08\\
23	4.60666019861035e-08\\
24	7.44926169987816e-08\\
25	1.92604946087505e-07\\
26	5.7615640936168e-07\\
27	1.91815904271328e-06\\
28	4.36807818437804e-06\\
29	4.11422513976415e-05\\
};

\addplot 
  table[row sep=crcr]{%
1	9.76192945828142e-11\\
2	7.74290958601504e-11\\
3	9.09824556143526e-11\\
4	7.14241527296628e-11\\
5	9.50982857833864e-11\\
6	2.36494805484366e-10\\
7	5.15209414574052e-10\\
8	4.97927806475785e-10\\
9	6.91237147304654e-10\\
10	1.33211254713407e-09\\
11	5.88184153722217e-10\\
12	7.04145905018673e-10\\
13	3.13908692401351e-10\\
14	2.05472517187893e-09\\
15	3.88763578567337e-09\\
16	1.93608033293828e-09\\
17	4.35411213767533e-09\\
18	5.36578717894911e-09\\
19	8.43168114849999e-09\\
20	6.15389883664158e-09\\
21	8.87247396150443e-09\\
22	2.28996100323389e-08\\
23	2.3811814254679e-08\\
24	4.55091659074949e-08\\
25	9.99314348756048e-08\\
26	3.64911771734583e-07\\
27	1.13216987256009e-06\\
28	5.60012815690102e-06\\
29	3.77533966156465e-05\\
};

\addplot 
  table[row sep=crcr]{%
1	4.85051246011334e-11\\
2	6.47384112875604e-11\\
3	8.53748642691232e-11\\
4	6.59101662796204e-11\\
5	1.21218687079376e-10\\
6	6.25926837973749e-11\\
7	1.23060864903884e-10\\
8	5.85894323633366e-10\\
9	2.17934700291632e-10\\
10	1.37531308223084e-09\\
11	1.07165691805532e-09\\
12	4.42126736562198e-09\\
13	8.8184947999937e-09\\
14	4.81837212024616e-09\\
15	1.1217538561537e-08\\
16	2.28970772377598e-08\\
17	2.9560250523353e-08\\
18	2.02881484057242e-08\\
19	5.40429181782971e-08\\
20	6.18909778719258e-08\\
21	7.4999431486013e-08\\
22	1.25202024948457e-07\\
23	2.51434925479816e-07\\
24	4.68553784420697e-07\\
25	9.77389205692627e-07\\
26	2.58459625012484e-06\\
27	5.04860021492855e-06\\
28	5.52571728287457e-05\\
29	6.51173024842941e-05\\
};

\addplot 
  table[row sep=crcr]{%
1	9.76192945828142e-11\\
2	7.74290958601504e-11\\
3	9.09824556143526e-11\\
4	7.14241527296628e-11\\
5	9.50982857833864e-11\\
6	8.04075165515394e-11\\
7	4.88050901009712e-11\\
8	4.97927366300287e-10\\
9	6.91236879159094e-10\\
10	1.3321124878972e-09\\
11	3.83865375411707e-09\\
12	7.04145324344549e-10\\
13	3.87802830408005e-09\\
14	2.05472511372293e-09\\
15	4.97413568825784e-09\\
16	1.03263994830749e-08\\
17	1.88671530378068e-08\\
18	2.0494088742831e-08\\
19	2.91397334642005e-08\\
20	3.75722349422693e-08\\
21	4.99497800793867e-08\\
22	7.25505908826546e-08\\
23	1.19541226732605e-07\\
24	2.08340177355875e-07\\
25	4.50195640484903e-07\\
26	1.49630354819933e-06\\
27	5.62605892476338e-06\\
28	1.94004075815875e-05\\
29	0.0002609282421792\\
};
\addplot [color=black,loosely dotted,forget plot]
  table[row sep=crcr]{1	1e-08\\
29	1e-08\\
};
\end{axis}
\end{tikzpicture}%
\caption{Residual norms for Example~1: Scaled computed residual norms $\fR_k^{\text{comp}}$ and inner residual norms $\|r_k^A\|$, $\|r_k^B\|$ against the outer iteration number for different dynamic stopping criteria. Only one curve for $\fR_k^{\text{comp}}$ is shown.}
\label{fig:ex1_res}
\end{figure}

\begin{figure}[t]
%
%
\definecolor{mycolor1}{rgb}{1.00000,0.00000,1.00000}%
\begin{tikzpicture}

\begin{axis}[%
width=0.6\linewidth,
height=0.5\linewidth,
cycle list name=res,
xmin=1,
xmax=30,
xlabel style={font=\color{white!15!black}},
xlabel={step},
ymin=0,
ymax=800,
ylabel style={font=\color{white!15!black}},
ylabel={sum(inner iters)},
axis background/.style={fill=white},
legend style={at={(1.1,0.9)}, anchor=north west, legend cell align=left, align=left, draw=white!15!black},
legend entries= {{it$^{\text{in},A}$, fixed}, {it$^{\text{in},B}$, fixed}, {it$^{\text{in},A}$, dyn.}, 
{it$^{\text{in},B}$, dyn.},{it$^{\text{in},A}$, dyn., BL}, {it$^{\text{in},B}$, dyn., BL}}]
\addplot 
  table[row sep=crcr]{%
1	33\\
2	63\\
3	90\\
4	113\\
5	133\\
6	150\\
7	164\\
8	176\\
9	186\\
10	195\\
11	204\\
12	215\\
13	228\\
14	242\\
15	258\\
16	276\\
17	296\\
18	317\\
19	340\\
20	365\\
21	391\\
22	418\\
23	445\\
24	472\\
25	498\\
26	523\\
27	545\\
28	564\\
29	581\\
};

\addplot 
  table[row sep=crcr]{%
1	33\\
2	66\\
3	98\\
4	129\\
5	158\\
6	185\\
7	209\\
8	231\\
9	250\\
10	267\\
11	281\\
12	294\\
13	306\\
14	319\\
15	334\\
16	351\\
17	371\\
18	393\\
19	418\\
20	445\\
21	475\\
22	507\\
23	540\\
24	574\\
25	608\\
26	642\\
27	675\\
28	707\\
29	738\\
};

\addplot 
  table[row sep=crcr]{%
1	27\\
2	53\\
3	77\\
4	99\\
5	118\\
6	134\\
7	148\\
8	159\\
9	169\\
10	177\\
11	186\\
12	197\\
13	209\\
14	223\\
15	238\\
16	255\\
17	273\\
18	293\\
19	314\\
20	336\\
21	359\\
22	382\\
23	403\\
24	423\\
25	441\\
26	454\\
27	462\\
28	468\\
29	471\\
};

\addplot 
  table[row sep=crcr]{%
1	27\\
2	53\\
3	78\\
4	102\\
5	124\\
6	143\\
7	159\\
8	173\\
9	185\\
10	195\\
11	204\\
12	212\\
13	220\\
14	228\\
15	238\\
16	249\\
17	261\\
18	275\\
19	290\\
20	306\\
21	324\\
22	343\\
23	363\\
24	383\\
25	402\\
26	420\\
27	436\\
28	450\\
29	462\\
};

\addplot 
  table[row sep=crcr]{%
1	27\\
2	53\\
3	77\\
4	99\\
5	118\\
6	135\\
7	149\\
8	160\\
9	170\\
10	178\\
11	187\\
12	197\\
13	208\\
14	221\\
15	235\\
16	250\\
17	266\\
18	284\\
19	303\\
20	323\\
21	343\\
22	363\\
23	382\\
24	399\\
25	413\\
26	422\\
27	428\\
28	433\\
29	436\\
};

\addplot 
  table[row sep=crcr]{%
1	27\\
2	53\\
3	78\\
4	102\\
5	124\\
6	144\\
7	162\\
8	176\\
9	188\\
10	198\\
11	207\\
12	215\\
13	222\\
14	230\\
15	239\\
16	249\\
17	260\\
18	273\\
19	287\\
20	302\\
21	318\\
22	335\\
23	352\\
24	369\\
25	386\\
26	401\\
27	415\\
28	427\\
29	437\\
};

\end{axis}

\end{tikzpicture}%
\caption{Inner iteration numbers for Example~1: cumulative sum of the inner iteration steps against the outer iteration number  
for fixed inner tolerances and different dynamic stopping criteria.}
\label{fig:ex1_innit}
\end{figure}
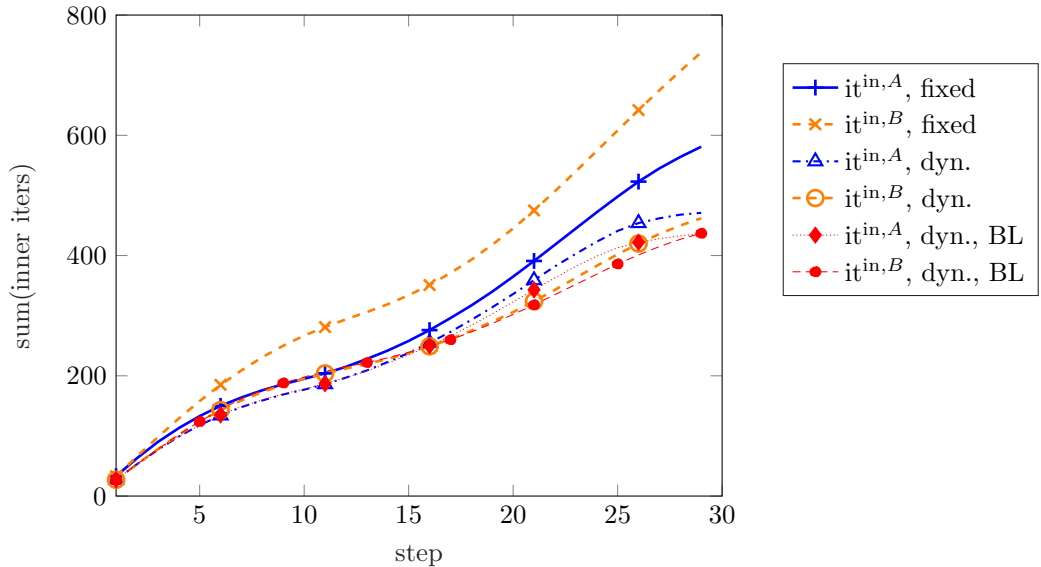

For Example~2 we can make similar observation from the data in Table~\ref{tab:ex_results}. Using the dynamic stopping criteria leads to savings in the compung times of roughly 30\%.  Figure~\ref{fig:ex2_res} illustrates again the history of the scaled computed Sylvester residual norm $\fR_k^{\text{comp.}}$ and the inner residual norms $\|r_k^A\|$, $\|r_k^A\|$. For Example~2, the strategy [dynamic, BL, $B$] actually leads to a very small reduction in the computing time due to a small change in the iteration numbers $\sum$it$^{\text{in},A}$, $\sum$it$^{\text{in},B}$. This is also visible in Figure~\ref{fig:ex2_res}, where the inner residual norms $\|r^A\|$ are in this variant slightly larger but the $\|r_k^B\|$ (corresponding to the much smaller linear systems) are kept at a much lower level. 

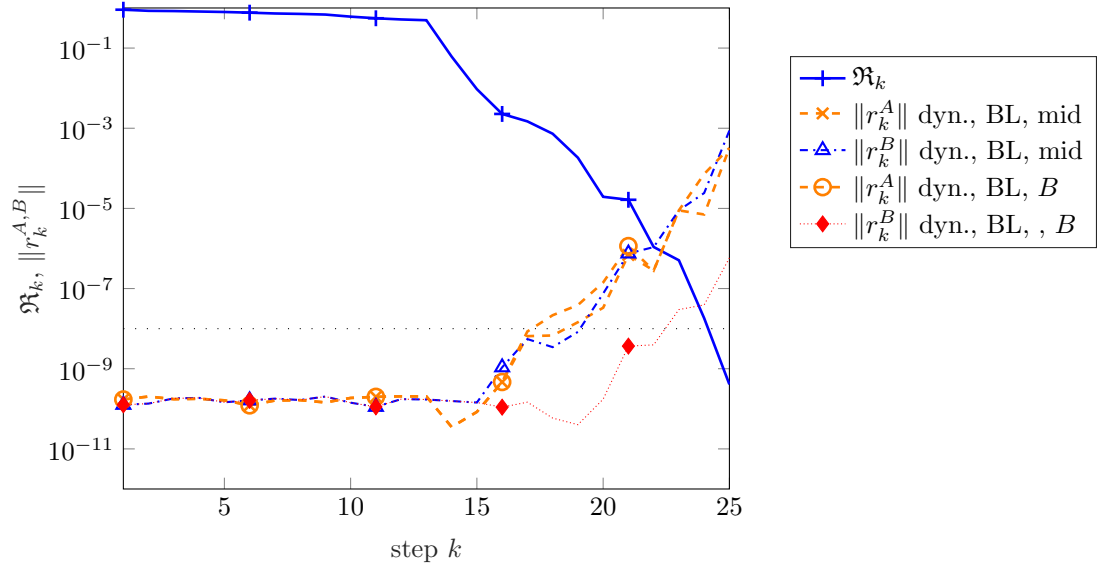
\begin{figure}[t]
%
%
\begin{tikzpicture}

\begin{axis}[%
width=0.6\linewidth,
height=0.5\linewidth,
cycle list name=res,
xmin=1,
xmax=25,
xlabel style={font=\color{white!15!black}},
xlabel={step $k$},
ylabel={$\fR_k$, $\|r_k^{A,B}\|$},
ymode=log,
ymin=1e-12,
ymax=1,
yminorticks=true,
ylabel style={font=\color{white!15!black}},
axis background/.style={fill=white},
legend style={legend cell align=left, align=left, draw=white!15!black},
axis background/.style={fill=white},
legend style={at={(1.1,0.9)}, anchor=north west, legend cell align=left, align=left, draw=white!15!black},
legend entries= {$\fR_k$, {$\|r_k^A\|$ dyn., BL, mid}, {$\|r_k^B\|$ dyn., BL, mid},{$\|r_k^A\|$ dyn., BL, $B$}, {$\|r_k^B\|$ dyn., BL, , $B$}}
]
\addplot 
  table[row sep=crcr]{%
1	0.905949848939738\\
2	0.84925747556916\\
3	0.837934254622744\\
4	0.815875767090492\\
5	0.794687432790681\\
6	0.764332046777367\\
7	0.729494431137666\\
8	0.711427577943622\\
9	0.685410723908695\\
10	0.60796007575727\\
11	0.553904648607959\\
12	0.517678109257171\\
13	0.494689428406385\\
14	0.0613421603292294\\
15	0.00943930287775688\\
16	0.00227264357844489\\
17	0.00148197227999537\\
18	0.000729479166013218\\
19	0.000184133403092876\\
20	1.94832840045619e-05\\
21	1.6411970260935e-05\\
22	1.09056102967108e-06\\
23	5.07404117261671e-07\\
24	1.82664615855496e-08\\
25	4.06923813387215e-10\\
};

\addplot 
  table[row sep=crcr]{%
1	1.71183266485668e-10\\
2	2.0264020500727e-10\\
3	1.74069389728628e-10\\
4	1.78120779173924e-10\\
5	1.62886583368341e-10\\
6	1.21994823508787e-10\\
7	1.61085854100628e-10\\
8	1.64429963265045e-10\\
9	1.44340412001233e-10\\
10	1.8675598914033e-10\\
11	1.98951705284965e-10\\
12	2.05025125815512e-10\\
13	2.04275208748937e-10\\
14	3.51806923073292e-11\\
15	8.38282708301924e-11\\
16	4.66567882473601e-10\\
17	6.50043817360944e-09\\
18	6.75552372372697e-09\\
19	1.44843112219257e-08\\
20	3.32087320816809e-08\\
21	6.89301115441349e-07\\
22	2.82709099102681e-07\\
23	8.9263766816946e-06\\
24	7.04883283133642e-06\\
25	0.000321252913354224\\
};

\addplot 
  table[row sep=crcr]{%
1	1.25813359779649e-10\\
2	1.35085831877932e-10\\
3	1.8220655051018e-10\\
4	1.87094597347677e-10\\
5	1.4575256752865e-10\\
6	1.62380843307784e-10\\
7	1.79341668323564e-10\\
8	1.65891656046914e-10\\
9	2.0235362099889e-10\\
10	1.42285276332391e-10\\
11	1.11933593884355e-10\\
12	1.74248479506592e-10\\
13	1.70857695270897e-10\\
14	1.56422187687165e-10\\
15	1.43456630933445e-10\\
16	1.08468631973708e-09\\
17	5.49557028630285e-09\\
18	3.45576650719637e-09\\
19	8.29340363585243e-09\\
20	7.51101957016964e-08\\
21	7.4981112109161e-07\\
22	1.08267845487971e-06\\
23	8.96053956841324e-06\\
24	2.41522582787191e-05\\
25	0.000876022487806166\\
};

\addplot 
  table[row sep=crcr]{%
1	1.71183266485668e-10\\
2	2.0264020500727e-10\\
3	1.74069389728628e-10\\
4	1.78120779173924e-10\\
5	1.62886583368341e-10\\
6	1.21994823508787e-10\\
7	1.61085854100628e-10\\
8	1.64429963265045e-10\\
9	1.44340412001233e-10\\
10	1.8675598914033e-10\\
11	1.98951705284965e-10\\
12	2.05025125815512e-10\\
13	2.04275208748937e-10\\
14	3.51806923073292e-11\\
15	8.38282708301924e-11\\
16	4.66567882473601e-10\\
17	8.51495444606341e-09\\
18	2.16066207265827e-08\\
19	3.87945282015891e-08\\
20	1.42036634184921e-07\\
21	1.14770678276554e-06\\
22	2.82709141477852e-07\\
23	8.94297221166888e-06\\
24	7.21558758374942e-05\\
25	0.000320861517254436\\
};

\addplot 
  table[row sep=crcr]{%
1	1.25813359779649e-10\\
2	1.35085831877932e-10\\
3	1.8220655051018e-10\\
4	1.87094597347677e-10\\
5	1.4575256752865e-10\\
6	1.62380843307784e-10\\
7	1.79341668323564e-10\\
8	1.65891656046914e-10\\
9	2.0235362099889e-10\\
10	1.42285276332391e-10\\
11	1.11933593884355e-10\\
12	1.74248479506592e-10\\
13	1.70857695270897e-10\\
14	1.56422187687165e-10\\
15	1.43456630933445e-10\\
16	1.09869904675494e-10\\
17	1.46476813970356e-10\\
18	5.87715604735357e-11\\
19	4.01744419689191e-11\\
20	1.74219899204921e-10\\
21	3.66727520133086e-09\\
22	3.94103130243168e-09\\
23	2.99233783200009e-08\\
24	3.91360366062511e-08\\
25	5.81126777046091e-07\\
};
\addplot [color=black,loosely dotted,forget plot]
  table[row sep=crcr]{1	1e-08\\
29	1e-08\\
};
\end{axis}

\end{tikzpicture}%
\caption{Residual norms for Example~2: Scaled computed residual norms $\fR_k^{\text{comp}}$ and inner residual norms $\|r_k^A\|$, $\|r_k^A\|$ against the outer iteration number for different dynamic stopping criteria.}
\label{fig:ex2_res}
\end{figure}

In Example~3 only the sequence of linear systems defined by $A$ is solved iteratively, while direct solvers are used for the other one with $B$. Hence, only a threshold for $\|r^A_k\|$ is set by simple setting $\|r_k^B\|=0$ in~\eqref{ADI_relax2_prak}. The results in Table~\ref{tab:ex_results} show again reduced numbers of inner iterations gained with the dynamic stopping strategies and savings of approximately 25\% in the computing time.

The results for the generalized Sylvester equation of Example~4 are in line with those of the previous examples. Here, the dynamic criteria lead runtime savings between roughly 33\% and 43\%. The strategy [dynamic, BL, $B$] does seem to slightly slow down the LR-ADI iteration as five more outer steps are needed, but overall the total number of inner iteration steps as well as the computing time is lower compared to using fixed inner tolerances. We also see that the inexact LR-ADI with fixed tolerances does actually take slightly more time than the exact ADI iteration. The arising linear systems in Example~4 seem to be harder for the inner solver (MINRES) compared to the other examples. That is the reason why we opted to update the preconditioners in every step. Using smaller drop tolerances for the incomplete Cholesky factorization did not lead to significant improvements. 

A further noteworthy observation in these and other experiments (not reported here) is that sometimes in the inexact ADI iteration with fixed tolerances, the required residual thresholds could not or hardly be achieved by the inner solver. This was much less frequently an issue with the proposed dynamic criteria because there the smallest inner residual norms are typically only required in the first outer iteration steps.  

\section{Conclusions and future research perspectives}\label{sec:concl}
We considered the inexact low-rank ADI iteration for large-scale Sylvester equations and proposed dynamic stopping criteria for the inner solvers which are used to iteratively solve to arising linear systems. We provided theoretical results showing that, with an appropriate choice for the inner accuracies, the residuals in inexact ADI iteration are only a small perturbation of the exact ADI iteration. Moreover, the practical implementation of the dynamic stopping criteria was discussed and numerical experiments confirmed the effectiveness of these strategies, leading the less inner iteration steps and, hence, smaller computing times compared to the case when constant inner tolerances were used. 

A potential next research direction might be subspace recycling techniques for the the sequences of shifted linear systems by, e.g., storing the Krylov
basis obtained from solving one linear system. This was, e.g., discussed for the Lyapunov LR-ADI in~\cite{morLi00} and a similar idea was developed in~\cite{BenDS23} leading to an efficient hybrid method. Corresponding strategies for the Sylvester-ADI iteration are a promising future research. Conducting similar studies for the linear systems inside rational Krylov projection methods for Sylvester equations~\cite{PalS18} is, of course, also a further worthwhile research direction. 

\end{document}